 \documentclass{amsart}
 \date{March 31, 2008}   

 \usepackage{amsmath,amsfonts,amssymb,amsthm}
 \usepackage{enumerate,color}
\usepackage[applemac]{inputenc}

\newcommand{\ds}{\displaystyle}

\theoremstyle{plain}
\newtheorem{theorem}{Theorem}[section]
\newtheorem{proposition}[theorem]{Proposition}
\newtheorem{corollary}[theorem]{Corollary}
\newtheorem{lemma}[theorem]{Lemma}
\newtheorem{definition}{Definition}[section]

\numberwithin{equation}{section}

\DeclareMathOperator{\BD}{BD}
\DeclareMathOperator{\QC}{QC}
\renewcommand{\r}{\mathbb{R}}%
\newcommand{\tq}{\, \big| \, }%

\title{Distortion of Mappings and $L_{q,p}$-Cohomology}
\author{Vladimir Gol'dshtein}


\address{Department of Mathematics, Ben Gurion University, P.O.Box 653, Beer Sheva, 
84105, Israel}
\email{vladimir@bgu.ac.il}

\author{Marc Troyanov}
\address{M. Troyanov, Section de Math{\'e}matiques,  \'Ecole Polytechnique F{\'e}d\'erale de
Lausanne, 1015 Lausanne - Switzerland}
\email{marc.troyanov@epfl.ch}

\begin{document}
 \maketitle
 
 \begin{abstract}
We study some relation between some geometrically defined classes of diffeomorphisms between
manifolds and the $L_{q,p}$-cohomology of these manifolds. Some applications to vanishing and non vanishing results in $L_{q,p}$-cohomology are given.
\end{abstract}


\medskip

\section{Introduction}

The $L_{q,p}$-cohomology is an invariant of  Riemannian manifolds defined to be the quotient of the space of $p$-integrable closed differential $k$-forms on the manifold modulo the exact forms having a $q$-integrable primitive:
$$
 H^k_{q,p}(M) = \{\omega \tq \omega \text{ is a $k$-form, } |\omega |\in L^p(M) \text{ and } d\omega = 0\}/\{d\theta \tq  |\theta |\in L^q(M) \}.
$$
This invariant has been first defined for the special case $p=q=2$ in the 1970's and has been intensively studied since then, we refer to the book \cite{luck2002} for an overview of $L_2$-cohomology.  The  $L_{q,p}$-cohomology 
has been introduced in the early 1980's as an invariant of the Lipschitz structure of manifolds, see  \cite{GK1}. During the next two decades, the main interest was focused on the case $p=q$, i.e on $L_p$-cohomology, and the last chapter of the book \cite{Grom} by M. Gromov is devoted to this subject ; see also \cite{GK2,Grom,pansu99,pansu2002,pansu2007,pansu2008} for more geometrical applications of $L_p$-cohomology.

Although the $L_{q,p}$-cohomology with $q\neq p$ has attracted less attention, it posses  a richer structure. The subject is also motivated by  its connections with Sobolev type inequalities \cite{GT2006} and quasiconformal geometry  \cite{GT2007}.
See also \cite{GK2,SOL,kopylov2007,kopylov2008}   for other results on $L_{q,p}$-cohomology.

\medskip

When an invariant of a geometric object has been defined, it is important to investigate its functorial properties, i.e. its behavior under various classes of mappings. It is one of our goal in the present paper to describe a natural class of maps which induces morphisms at the level of $L_{q,p}$-cohomology. Our answer is restricted to  the case of diffeomorphisms and is given in Theorem 
\ref{thm:CA}(C) below.  

\medskip

A diffeomorphisms will behave functorially for $L_{q,p}$-cohomology, if its \emph{distortion} is controlled in  some specific way. To explain what is meant by the distortion, consider a diffeomorphism $f : M \to \tilde{M}$ between two Riemannian manifolds. On then define for any $k$ the \emph{principal invariant }of $f$ as
\[
 \sigma_{k}(f,x) =\sum_{i_{1}<i_{2}<\cdots<i_{k}}\lambda_{i_{1}}(x)\lambda_{i_{2}}(x)\cdots\lambda_{i_{k}}(x),
\]
where the $\lambda_i$'s are the principal invariants of $df_x$, i.e. the eigenvalues of $\sqrt{(df_x)^*(df_x)}$. One then  say that $f$ has \emph{bounded $(s,t)$-distortion in degree $k$}, and we write
$
  f \in \BD^k_{(s,t)}(M,\tilde{M}),
$
if
$$
  \left(\sigma_{k}(f,x)\right)^{s} \cdot J_{ f}^{-1}(x) \in L^t(M)
$$
where $J_f$ is the Jacobian of $f$.

\medskip

The class $BD^1_{n,\infty}$ (where $n$ is the dimension of $M$) is exactly the class of quasiconformal diffeomorphisms (also called mappings with bounded distortion), which has been introduced by Y. Reshetnyak in the early 1960's and has been intensively studied since then. The classes $BD^1_{s,\infty}$ has been studied by different authors and under various names, see
 \cite{ferrand,Gaf,GGR,Maz85,Masha,pansu97,Reiman,TV2002,Vdp2000,VU}. The class $BD^{n-1}_{s,\infty}$ also appears in \cite{TV2002}, where some obstructions are given.

\medskip

As a preliminary step to the study of functoriality in $L_{q,p}$-cohomology, 
we study diffeomorphisms $f:M \to \tilde{M}$ that induces bounded operator between the 
Banach spaces of $\tilde{p}$-integrable differential $k$-forms. The result is formulated in Proposition
\ref{pro:prBMD1} : it states that \emph{ a  diffeomorphism $f \in \BD^k_{(\tilde{p},t)}(M,\tilde{M})$ induces a bounded operator $f^{*}:L^{\tilde{p}}(\tilde{M},\Lambda^{k})\rightarrow L^{p}(M,\Lambda^{k})$  if
$p\leq\tilde{p} < \infty$ and $t = \frac{p}{\tilde{p}-p}$.
}
Let us note that finer information are available in the case $k = 1$, see \cite{GGR, GRo,VG76,VU}.

\medskip

To obtain a functoriality in $L_{q,p}$-cohomology, we need to control the distortion of the map $f$
both on $k$-forms and  on $(k-1)$ forms. This is formulated in Theorem \ref{thm:CA}(C), which states in particular that 
 \emph{ a  diffeomorphism $f\in \BD_{(\tilde{q}',r)}^{n-k+1}(M, \tilde{M}) \cap \BD^k_{(\tilde{p},t)}(M,\tilde{M})$  induces a well defined linear map 
 $ f^{*}:H_{\tilde{q},\tilde{p}}^{k}(\tilde{M})\rightarrow H_{q,p}^{k}(M)$ if 
$p\leq\tilde{p}$, $q\leq \tilde{q}$, $t = \frac{p}{\tilde{p}-p}$, $u = \frac{q}{\tilde{q}-q}$ and 
$\tilde{q}' = \frac{\tilde{q}}{\tilde{q}-1}$.
}

\medskip

This is a quite technical result, and it would be nice to be able to give conditions under which the
map $f^{*}$ is injective at the level of $L_{q,p}$-cohomology. But unfortunately, the results we give in section 5 strongly suggest that it will be hard or impossible to find conditions for injectivity, except for the special cases of  quasiconformal or bilipschitz maps. However we have the following result (theorem \ref{thm:CA}(B)), which allows us to prove some vanishing results in $L_{q,p}$-cohomology without requiring the functoriality :  
\emph{If there exists a diffeomorphism
  $f\in \BD_{(\tilde{q}',r)}^{n-k+1}(M, \tilde{M}) \cap \BD^k_{(\tilde{p},t)}(M,\tilde{M})$
  with 
  $p\leq\tilde{p}$, $q\leq \tilde{q}$, $t = \frac{p}{\tilde{p}-p}, r =  \frac{q(\tilde{q}-1)}{q-\tilde{q}}$,  and 
$\tilde{q}' = \frac{\tilde{q}}{\tilde{q}-1}$,
 then   $H_{q,p}^{k}(M) = 0$ implies $H_{\tilde{q},\tilde{p}}^{k}(\tilde{M})= 0$.
} 
We give two concrete examples showing  how this result can be used to prove vanishing and non vanishing results in $L_{q,p}$-cohomology. 

\medskip

The paper is organized as follows. In section 2, we recall the definition of $L_{q,p}$-cohomology and some known facts about the distortion of linear maps. In section 3, we discuss the effect of a diffeomorphism $f$ at the level of $L_{q,p}$-cohomology, assuming that the map $f$ induces 
bounded operators at the level of some Banach spaces of integrable differential forms ; these are abstract results. In section 4, we introduce the class of diffeomorphisms with bounded $(s,t)$-distortion and in section 5, we relate these diffeomorphisms with quasiconformal and bilipshitz maps. Section 6 contains our main results, which relates the distortion of diffeomorphisms to 
$L_{q,p}$-cohomology and in section 7, we give two concrete applications of these results.  In the last section, we shortly discuss our smoothness restrictions.

\section{Preliminary notions}

\subsection{$L_{q,p}$-cohomology}

We shortly recall the definition of $L_{q,p}$-cohomology, referring to the paper \cite{GT2006}
for more details. Let $M$ be an oriented  Riemannian manifold, we denote by 
$C_{c}^{\infty}(M,\Lambda^{k})$ the vector space of
smooth differential forms of degree $k$ with compact support on $M$
and by $L_{loc}^{1}(M,\Lambda^{k})$ the space of differential $k$-forms
whose coefficients (in any local coordinate system) are locally integrable.

The form $\theta\in L_{loc}^{1}(M,\Lambda^{k})$ is said to be the
\emph{weak exterior differential} of $\phi\in L_{loc}^{1}(M,\Lambda^{k-1})$,
and one writes $d\phi=\theta$, if for each $\omega\in C_{c}^{\infty}(M,\Lambda^{n-k})$,
one has 
\[
\int_{M}\theta\wedge\omega=(-1)^{k}\int_{M}\phi\wedge d\omega\,.
\]

Let $L^{p}(M,\Lambda^{k})$ be the Banach space of differential forms in
$L_{loc}^{1}(M,\Lambda^{k})$ such that 
\[
\Vert\theta\Vert_{p}:=\left(\int_{M}|\theta|^{p}dx\right)^{\frac{1}{p}}<\infty\,.
\]
We denote by  $Z_{p}^{k}(M)$ the space of weakly closed forms in $L^{p}(M,\Lambda^{k})$, 
i.e. $Z_{p}^{k}(M)=L^{p}(M,\Lambda^{k})\cap\ker d$. It is a closed subspace. We also define
\[
B_{q,p}^{k}(M):=d\left(L^{q}(M,\Lambda^{k-1})\right)\cap L^{p}(M,\Lambda^{k}),
\]
this is the space of  exact forms in $L^p$ having a primitive in $L^q$ and we have  $B_{q,p}^{k}(M)\subset Z_{p}^{k}(M)$, because $d\circ d=0$.

\begin{definition}
The $L_{q,p}$\emph{-cohomology} of $(M,g)$ (where $1\leq p,q\leq\infty$)
is defined to be the quotient 
\[
H_{q,p}^{k}(M):=Z_{p}^{k}(M)/B_{q,p}^{k}(M)\,,
\]
 and the \emph{reduced} $L_{q,p}$\emph{-cohomology} of $(M,g)$ is
\[
\overline{H}_{q,p}^{k}(M):=Z_{p}^{k}(M)/\overline{B}_{q,p}^{k}(M)\,,
\]
where $\overline{B}_{q,p}^{k}(M)$ is the closure of $B_{q,p}^{k}(M)$.
\end{definition}
The reduced cohomology is naturally a Banach space. 
When $p=q$, we simply speak of $L_{p}$-cohomology and write $H_{p}^{k}(M)$
and $\overline{H}_{p}^{k}(M)$.

\subsection{Linear map between Euclidean spaces}

Recall that an Euclidean vector space $(E,g)$ is a finite dimensional
real vector space equipped with a scalar product. Two linear mappings
$A,B\in L(E_{1};E_{2})$ between two Euclidean vector spaces $(E_{1},g_{1})$ of dimension
$n$ and $m$ are said to
be \emph{orthogonally equivalent} if there exist orthogonal transformations
\ $Q_{1}\in O(E_{1})$ and \ $Q_{2}\in O(E_{2})$ such that \ $B=Q_{2}^{-1}AQ_{1}$,
i.e. the diagram 
\[
\begin{array}{lll}
\quad E_{1} & \overset{A}{\rightarrow} & E_{2}\\
{\tiny Q}_{1}\uparrow &  & \uparrow{\tiny Q}_{2}\\
\quad E_{2} & \overset{B}{\rightarrow} & E_{2}\end{array}
\]
commutes. 
Given a linear mapping $\,\, A:(E_{1},g_{1})\rightarrow(E_{2},g_{2})$,
its (right) \emph{Cauchy-Green tensor}\ $\mathbf{c}$\  is the symmetric
bilinear form on $E_{1}$ defined by \ $\mathbf{c}(x,y)=g_{2}(Ax,Ay)$. 
The \emph{adjoint} of $A$ is the linear map \ $A^{\#}:E_{2}\rightarrow E_{1}$
satisfying 
\[
g_{2}(x,Ay)=g_{1}(A^{\#}x,y)
\]
for all \ $x\in E_{1}$ and $y\in E_{2}$. The Cauchy-Green tensor
and the adjoint are related by
 \[
\mathbf{c}(x,y)=g_{2}(Ax,Ay)=g_{1}(A^{\#}Ax,y).
\]

Let us denote the eigenvalues of $A^{\#}A$ by $\mu_{1},\mu_{2},...,\mu_{n}$. Then $\mu_{i}\in[0,\infty)$, for all $i$, and there exists orthonormal
basis $e_{1},e_{2},\cdots,e_{n}$ of $E_{1}$ and $e_{1}^{\prime},e_{2}^{\prime},\cdots,e_{m}^{\prime}$ of $E_{2}$ such that \ $Ae_{i}=\sqrt{\mu_{i}}e_{i}^{\prime}$ for all $i$. The matrix of $A^{\#}A$ with respect to an orthonormal basis $e_{1},e_{2},\cdots,e_{n}$ of $E_{1}$ 
coincides with the matrix $\mathbf{C}$ of the Cauchy-Green tensor
$\mathbf{c}$ in the same basis.

 \begin{definition} \ 
 The numbers \ $\lambda_{i}=\sqrt{\mu_{i}}$ are called the \emph{principal distortion coefficients} of $A$ or the \emph{singular values} of $A$.
 \end{definition}
 
The principal distortion coefficients can be computed from the distortion polynomial which is defined as follows:

\begin{definition}
Given an arbitrary basis \ $e_{1},e_{2},\cdots,e_{n}$ of $E_{1}$,
we associate to $g_{1}$ and $\mathbf{c}$, the $n\times n$ matrices
\ $\mathbf{G}=\left(g_{1}(e_{i},e_{j})\right)$ and $\mathbf{C}=\left(\mathbf{c}(e_{i},e_{j})\right)$.
The \emph{distortion polynomial}  of \textbf{\ {}}$A$ is the polynomial
\[
P_{A}(t):=\frac{\det(\mathbf{C}-t\mathbf{G})}{\det\mathbf{G}}.
\]
\end{definition}

The distortion polynomial $P_{A}(t)$ 
is independent of the choice of the basis $\left\{ e_{i}\right\} $, it coincides with the characteristic
polynomial of  $AA^{\#}$ and has nonnegative roots. In particular, the roots of $P_{A}$ are the eigenvalues $\mu_i$ of \ $AA^{\#}$ and the   $\lambda_{i}=\sqrt{\mu_{i}}$ are  the  principal distortion coefficients of $A$ and the distortion polynomial can thus be written in terms of
the principal distortion coefficients as 
\[
P_{A}(t)=\prod_{i}(t-\lambda_{i}^2).
\]

The following notion is also useful: 
\begin{definition}
The  \emph{principal invariants of $A$}  are  the  elementary symmetric polynomials in the $\lambda_{i}$'s, 
i.e. they are defined by 
$\sigma_{0}(A) = 1$ and 
\[
\sigma_{k}(A) =\sum_{i_{1}<i_{2}<\cdots<i_{k}}\lambda_{i_{1}}\lambda_{i_{2}}\cdots\lambda_{i_{k}}
\]
for $k = 1,\dots, 2 \dots, n$.
\end{definition}

\medskip

The following result is well known, it can be found e.g. in (\cite{resh}, page 57)

\begin{proposition}\label{pro.inv}
Two linear mappings $A,B\in L(E_{1};E_{2})$  are orthogonally equivalent
if and only if  they have the same principal invariants : 
$\sigma_{k}(A) =\sigma_{k}(B)$ for $k = 1,2,\dots, n$.
\end{proposition}

\bigskip

The  principal invariants of $A$ are related to the action of $A\in L(E_{1};E_{2})$ on the exterior algebras:
Recall that if $E$ is an Euclidean vector space, then the exterior
algebra $\Lambda E$ is equipped with a canonical scalar product. 
If $e_{1},e_{2},\cdots,e_{n}$ is an orthonormal basis of $E_{1}$,
then the $\binom{n}{k}$ multi-vectors $\left\{ e_{i_{1}}\wedge e_{i_{2}}\wedge\cdots\wedge e_{i_{k}}\right\} $
$\left(i_{1}<i_{2}<\cdots<i_{k}\right)$
form an orthonormal basis of $\Lambda^{k}E$. 

To any linear map $A\in L(E_{1};E_{2})$ we associate 
a linear map $\Lambda^{k}A\in L(\Lambda^{k}E_{1};\Lambda^{k}E_{2})$,
and we have 
\begin{equation}
\frac{1}{\binom{n}{k}}\sigma_{k}\leq\left\Vert \Lambda^{k}A\right\Vert \leq\sigma_{k}\label{eq:l1}
\end{equation}
Indeed, suppose that $\lambda_{1}\leq\lambda_{2}\leq\cdots\leq\lambda_{n}$
are the principal distortion coefficients of $A$, then we have \ $\left\Vert \Lambda^{k}A\right\Vert =\lambda_{n-k+1}\lambda_{n-k+2}\cdots\lambda_{n}$
and $\sigma_{k}:=\sum_{i_{1}<i_{2}<\cdots<i_{k}}\lambda_{i_{1}}\lambda_{i_{2}}\cdots\lambda_{i_{k}}$. 

\bigskip

If $E_1=E_2= \r^n$ and $A$ is a diagonal matrix with nonnegative entries, then we have
\[
\sigma_{k}=\mathrm{Trace}(\Lambda^{k}A).
\]

\bigskip
 
The principal distortion coefficients also have the following geometric
interpretation: 

\begin{itemize}
\item[$\circ$] If \ $S\subset E_{1}$ is the unit ball, then \ $A(S)\subset E_{2}$
is an ellipsoid contained in  $\mathrm{Im} A$ and  whose principal axis
are the non vanishing $\lambda_{i}$. 
\item[$\circ$] Suppose $\dim (E_1) = \dim (E_2) = n$. The \emph{Jacobian} $J_{A}:=\sigma_{n}=\lambda_{1}\lambda_{2}...\lambda_{n}$ measures the volume distortion. 
\item[$\circ$]  If $\dim (E_1) = \dim (E_2) = n$ and $A$ is invertible, then the principal distortion coefficients
of $A^{-1}$ are the inverse of the principal distortion coefficients of $A$.
\item[$\circ$]  The norm of $A$ as a linear operator is $\left\Vert A\right\Vert =\max_{v\neq0}\frac{\left\Vert Av\right\Vert }{\left\Vert v\right\Vert }=\max_{i}\lambda_{i}$. 
\end{itemize}

\medskip

\begin{lemma}\label{lem:lin1}  
If $\dim (E_1) = \dim (E_2) = n$ and $A$ is invertible, then for any $0\leq m\leq n$, we have
\[
\sigma_{m}(A^{-1})  =\frac{\sigma_{n-m}(A)}{J_{A}}.
\]
\end{lemma}

\textbf{Proof} Use the fact the principal distortion coefficients
of $A^{-1}$ are the inverse of the principal distortion coefficients of $A$, and compute.

\qed

\section{Diffeomorphism and $L_{q,p}$-cohomology}

Let $(M,g)$ and $(\tilde{M},\tilde{g})$ be two smooth oriented $n$-dimensional Riemannian
manifolds and $ f:M\rightarrow\tilde{M}$ be a diffeomorphism
such that the induced operator
\[
 f^{*}:L^{\tilde{p}}(\tilde{M},\Lambda^{k})\rightarrow L^{p}(M,\Lambda^{k})
 \]
 is bounded for some specified $p,\tilde{p}\in[0,\infty)$. Then the condition $ f^{*}d=d f^{*}$ implies
that 
\[
 f^{*}:Z_{\tilde{p}}^{k}(\tilde{M})\rightarrow Z_{p}^{k}(M)
\]
is a well defined  bounded operator. In the framework of $L_{q,p}$-cohomology there are two natural questions which then arise:
\begin{enumerate}[i.)]
  \item Suppose that  $\omega\in B_{\tilde{q},\tilde{p}}^{k}(\tilde{M})$. Under what conditions does this
  imply that $f^{*}\omega\in B_{q,p}^{k}(M)$, i.e. that 
  $$
   f^*(B_{\tilde{q},\tilde{p}}^{k}(\tilde{M}))  \subset  B_{q,p}^{k}(M) \ ?
  $$
\item Suppose that $ f^{*}\omega\in B_{q,p}^{k}(M)$. Under what conditions can we conclude that  $\omega\in B_{\tilde{q},\tilde{p}}^{k}(\tilde{M})$, i.e. that 
  $$
   (f^{-1})^*(B_{q,p}^{k}(M))  \subset  B_{\tilde{q},\tilde{p}}^{k}(\tilde{M}) \ ?
  $$

\end{enumerate}
A positive answer to the first question gives us a well defined linear map
\[
  f^{*}:H_{\tilde{q},\tilde{p}}^{k}(\tilde{M})\rightarrow H_{q,p}^{k}(M),
\]
 and  a positive answer to both  questions implies the injectivity of this linear map.
 
 \medskip
 
In this section we give an answer to these questions in terms of boundedness
of  the operators $ f^{*},$ and $f_{*}:=\left( f^{-1}\right)^{*}$. We begin with the second question.

\begin{theorem}\label{thm:thA} 
Let $ f:M\rightarrow\tilde{M}$ be a diffeomorphism, $1 \leq p \leq \tilde{p} < \infty$ and  
$1 \leq \tilde{q} \leq q < \infty$. Assume that  both operators
\[
 f^{*}:L^{\tilde{p}}(\tilde{M},\Lambda^{k})\rightarrow L^{p}(M,\Lambda^{k}),
\quad \text{and} \quad
 f_{*}:L^{q}(M,\Lambda^{k-1})\rightarrow L^{\tilde{q}}(\tilde{M},\Lambda^{k-1})
\]
are bounded.
Then for any  $\omega\in Z_{\tilde{p}}^{k}(\tilde{M})$, we have $f^*\omega\in Z_{p}^{k}(M)$.
Furthermore, if $\left[ f^{*}\omega\right]=0$
in $H_{q,p}^{k}(M)$ then $\left[\omega\right]=0$ in $H_{\tilde{q},\tilde{p}}^{k}(\tilde{M})$
(thus $H_{q,p}^{k}(M) = 0 \Rightarrow H_{\tilde{q},\tilde{p}}^{k}(\tilde{M})= 0$).
\end{theorem}

\bigskip

\textbf{Remarks \ } We should not conclude that  $f^{*}:H_{\tilde{q},\tilde{p}}^{k}(\tilde{M})\rightarrow H_{q,p}^{k}(M)$ is an injective map, because this map is a priory not even well defined.
\medskip

\medskip

\begin{proof}
Choose $\omega\in Z_{\tilde{p}}^{k}(\tilde{M})$. Because $ f^{*}:L^{\tilde{p}}(\tilde{M},\Lambda^{k})\rightarrow L^{p}(M,\Lambda^{k})$ 
is a bounded operator, $f^{*}\omega\in L^{p}(M,\Lambda^{k})$,
and since  and $d (f^{*}\omega) = f^{*}d\omega = 0$ we have $f^{*}\omega\in Z_{p}^{k}(M)$. 
Suppose now that $\left[ f^{*}\omega\right]=0$
in $H_{q,p}^{k}(M)$, then $ f^{*}\omega\in B_{q,p}^{k}(M)$, that is  there exists $\theta\in L^{q}(M,\Lambda^{k-1})$ such
that $d\theta= f^{*}\omega$. But by the second hypothesis the
operator $ f_{*}:L^{q}(M,\Lambda^{k})\rightarrow L^{\tilde{q}}(\tilde{M},\Lambda^{k})$
is bounded and therefore $ f_{*}\theta\in L^{\tilde{q}}(\tilde{M},\Lambda^{k})$.
We then have
$$
\omega =  f_{*}\left( f^{*}\omega\right) = f_{*}d\theta = d\left( f_{*}\theta\right) \in B_{\tilde{q},\tilde{p}}^{k}(\tilde{M})
$$
Therefore $\left[\omega\right]=0$ in $H_{\tilde{q},\tilde{p}}^{k}(\tilde{M})$.
\end{proof}

The argument of the previous proof  is illustrated in the following commutative diagrams:
\[
\begin{array}{ccc}
\quad Z_{\tilde{p}}^{k}(\tilde{M}) & \overset{ f^{*}}{\longrightarrow} & Z^k_p(M)\\ 
{\tiny d}{\uparrow} &  & \uparrow{\tiny d}\\
\quad L^{\tilde{q}}(\tilde{M},\Lambda^{k-1}) & \overset{ f_{*}}{\longleftarrow} & L^{q}(M,\Lambda^{k-1})\end{array}
\qquad
\begin{array}{ccc}
\quad\omega & \overset{ f^{*}}{\longrightarrow} &  f^{*}\omega\\
{\tiny d}\uparrow &  & \uparrow{\tiny d}\\
\quad f_{*}\theta & \overset{ f_{*}}{\longleftarrow} & \theta\end{array}
\]

\bigskip 

The next result gives us sufficient conditions for a diffeomorphism to behave functorially  at the  $L_{q,p}$-cohomology level.
 
\begin{theorem}\label{thm:thB} 
Let $ f:M\rightarrow\tilde{M}$ be a diffeomorphism and $1 \leq p \leq \tilde{p} < \infty$ and  
$1 \leq q \leq \tilde{q} < \infty$. Assume that
\[
 f^{*}:L^{\tilde{p}}(\tilde{M},\Lambda^{k})\rightarrow L^{p}(M,\Lambda^{k}),
\quad \text{and} \quad
 f^{*}:L^{\tilde{q}}(\tilde{M},\Lambda^{k-1})\rightarrow L^{q}(M,\Lambda^{k-1})
 \]
are bounded operators.
Then 
\begin{enumerate}[a.)]
  \item $f^{*}:\Omega_{\tilde{q},\tilde{p}}^{k-1}(\tilde{M})\rightarrow\Omega_{q,p}^{k-1}(M)$
is a bounded operator, 
  \item $ f^{*}:H_{\tilde{q},\tilde{p}}^{k}(\tilde{M})\rightarrow H_{q,p}^{k}(M)$
is a well defined linear map,
  \item $ f^{*}:\overline{H}_{\tilde{q},\tilde{p}}^{k}(\tilde{M})\rightarrow \overline{H}_{q,p}^{k}(M)$
is a well defined bounded operator,
\end{enumerate}
\end{theorem}

\medskip

\begin{proof}
a) By definition $\omega\in\Omega_{\tilde{q},\tilde{p}}^{k-1}(\tilde{M})$
if $\omega\in L^{\tilde{q}}(\tilde{M},\Lambda^{k-1})$ and $d\omega\in L^{\tilde{p}}(\tilde{M},\Lambda^{k})$.
Because both operators $ f^{*}:L^{\tilde{p}}(\tilde{M},\Lambda^{k})\rightarrow L^{p}(M,\Lambda^{k})$,
$ f^{*}:L^{\tilde{q}}(\tilde{M},\Lambda^{k-1})\rightarrow L^{q}(M,\Lambda^{k-1})$
are bounded and $ f^{*}d\omega=d f^{*}\omega$ we obtain
that $ f^{*}\omega\in\Omega_{q,p}^{k-1}(M)$. The operator
$ f^{*}:\Omega_{\tilde{q},\tilde{p}}^{k-1}(\tilde{M})\rightarrow\Omega_{q,p}^{k-1}(M)$
is clearly bounded.

\medskip

b) The condition $ f^{*}d=d f^{*}$ and the boundedness of the operators
$ f^{*}:L^{\tilde{p}}(\tilde{M},\Lambda^{k})\rightarrow L^{p}(M,\Lambda^{k})$
implies that $ f^{*}\left(Z_{\tilde{p}}^{k}(\tilde{M})\right)\subset Z_{p}^{k}(M)$.
Using the boundedness of the operator $ f^{*}:\Omega_{\tilde{q},\tilde{p}}^{k}(\tilde{M})\rightarrow\Omega_{q,p}^{k}(M)$ and the condition $ f^{*}d=d f^{*}$ we see that
\[
 f^{*}\left(B_{\tilde{q},\tilde{p}}^{k}(\tilde{M})\right)= f^{*}\left(d\Omega_{\tilde{q},\tilde{p}}^{k-1}(\tilde{M})\right)=d f^{*}\left(\Omega_{\tilde{q},\tilde{p}}^{k-1}(\tilde{M})\right)\subset d\left(\Omega_{q,p}^{k-1}(M)\right)=B_{q,p}^{k}(M).
 \]
The inclusions 
\begin{equation}\label{incl.t1}
f^{*}\left(Z_{\tilde{p}}^{k}(\tilde{M})\right)\subset Z_{p}^{k}(M),
 \qquad
 f^{*}\left(B_{\tilde{q},\tilde{p}}^{k}(\tilde{M})\right)\subset B_{q,p}^{k}(M)
\end{equation}
imply that the linear map
\[
 f^{*}:H_{\tilde{q},\tilde{p}}^{k}(\tilde{M})=Z_{\tilde{p}}^{k}(\tilde{M})/B_{\tilde{q},\tilde{p}}^{k}(\tilde{M})\rightarrow Z_{p}^{k}(M)/B_{q,p}^{k}(M)=H_{q,p}^{k}(M)
 \]
is well defined.

\medskip

c) Using the inclusions (\ref{incl.t1}) and the continuity of the operator $ f^{*}:\Omega_{\tilde{q},\tilde{p}}^{k}(\tilde{M})\rightarrow\Omega_{q,p}^{k}(M)$, we have
\begin{equation}\label{incl.t2}
 f^{*}\left(\overline{B_{\tilde{q},\tilde{p}}^{k}(\tilde{M})}\right)\subset\overline{ f^{*}\left(B_{\tilde{q},\tilde{p}}^{k}(\tilde{M})\right)}\subset\overline{B_{q,p}^{k}(M)}.
\end{equation}
Therefore the operator
\[
 f^{*}:\overline{H}_{\tilde{q},\tilde{p}}^{k}(\tilde{M})=Z_{\tilde{p}}^{k}(\tilde{M})/\overline{B_{\tilde{q},\tilde{p}}^{k}(\tilde{M})}\rightarrow Z_{p}^{k}(M)/\overline{B_{q,p}^{k}(M)}=\overline{H}_{q,p}^{k}(M)
 \]
is well defined and bounded.

\end{proof}

Using the two previous theorems, we have the following result:

\begin{theorem}\label{thm:thC}   
Let $ f:M\rightarrow\tilde{M}$ be a diffeomorphism and $1 \leq p \leq \tilde{p} \leq \infty$ and  
$1 \leq \tilde{q} = q \leq \infty$.  Assume that the operator
$
 f^{*}:L^{\tilde{p}}(\tilde{M},\Lambda^{k})\rightarrow L^{p}(M,\Lambda^{k})
$
is bounded and that
$
 f^{*}:L^{q}(\tilde{M},\Lambda^{k-1})\rightarrow L^{q}(M,\Lambda^{k-1})
$
is an isomorphism of Banach spaces.
Then the linear map
\[
 f^{*}:H_{q,\tilde{p}}^{k}(\tilde{M})\rightarrow H_{q,p}^{k}(M)
\]
is well defined and injective.
\end{theorem}

\medskip 

The proof is immediate.

\qed

\bigskip

\begin{corollary}\label{cor.C}
Let $f : M \to \tilde{M}$ satisfying the hypothesis of the previous theorem. 
If $T_{q,p}^{k}(M)=0$ then $T_{\tilde {q},\tilde{p}}^{k}(\tilde{M})=0$.
\end{corollary}

\medskip

\textbf{Proof} Since $T_{q,p}^{k}(M)=0$, we have $\overline{B_{q,p}^{k}(M)} ={B_{q,p}^{k}(M)}$.
The hypothesis of  Theorem \ref{thm:thB} are satisfied, thus the inclusions (\ref{incl.t2}) holds and
we thus have
$$
 f^{*}\left(\overline{B_{\tilde{q},\tilde{p}}^{k}(\tilde{M})}\right)\subset\overline{B_{q,p}^{k}(M)}
 = {B_{q,p}^{k}(M)}.
$$
Choose now  an arbitrary element 
 $\omega\in \overline{B_{\tilde{q},\tilde{p}}^{k}(\tilde{M})}$. We have  $f^{*}\omega\in {B_{q,p}^{k}(M)}$
 by the previous inclusion, this means that  $[f^{*}\omega] = 0\in {H_{q,p}^{k}(M)}$, but $ f^{*}:H_{\tilde{q},\tilde{p}}^{k}(\tilde{M})\rightarrow H_{q,p}^{k}(M)$ is injective by the previous theorem and therefore
$[\omega] = 0$ in $H_{q,\tilde{p}}^{k}(\tilde{M})$, that is $\omega \in {B_{q,\tilde{p}}^{k}(\tilde{M})}$.
Since $\omega$ was arbitrary, we have shown that $\overline{B_{q,\tilde{p}}^{k}(\tilde{M})}={B_{q,\tilde{p}}^{k}(\tilde{M})}$, i.e. $T_{q,\tilde{p}}^{k}(\widetilde{M})=0$.

\qed

 \bigskip

 \textbf{Remark}  The hypothesis in Theorem  \ref{thm:thC} seem to be very restrictive, the results of section 5 suggest that it will be difficult to find diffeomorphisms satisfying these hypothesis and which aren't bilipshitz or quasiconformal.  See  the discussion at the end of section 5.

\section{Diffeomorphisms with controlled distortion.}

Let $(M,g)$ and $(\tilde{M},\tilde{g})$ be two smooth oriented Riemannian manifolds. 
In this section we study classes of diffeomorphisms $f:M\rightarrow\tilde{M}$ with bounded distortion of an integral type  that induce bounded operators 
$ f^{*}:L^{\tilde{p}}(\tilde{M},\Lambda^{k})\rightarrow L^{p}(M,\Lambda^{k})$ for $1\leq p \leq \tilde{p}  \leq \infty$. 
To define these classes we use the notation 
$$\sigma_{k}(f,x) = \sigma_{k}(d f_{x})$$
for the $k$-th principal invariant of the differential $df_x$. We also write $\sigma_{k}(f)$
when there is no risk of confusion, observe that $\sigma_{n}(f) = J_{f}$, where $J_f$ is the
Jacobian of $f$.

\medskip

\begin{definition}\label{def.BMD}
A diffeomorphism  $f:M\rightarrow\tilde{M}$ is said to be of \emph{bounded
$(s,t)$-distortion in degree $k$}, and we write
$
  f \in \BD^k_{(s,t)}(M,\tilde{M}),
$
if
$$
   \left(\sigma_{k}(f)\right)^{s} J_{ f}^{-1} \in L^t(M).
$$
It is assumed that $1\leq s < \infty$ and $0< t \leq \infty$.
\end{definition}

It is convenient to introduce the quantity 
\[
 K_{s,t,k}(f) = \left\Vert \frac {\left(\sigma_{k}(f)\right)^{s}}{ J_{ f}(x)} \right\Vert _{L^t(M)},
\]
the mapping $f$ belongs then to  $\BD^k_{(s,t)}(M,\tilde{M}),$ if and only if $K_{s,t,k}(f)<\infty$.

\medskip

\medskip

\begin{proposition}\label{pro:prBMD1}
Let   $f:M\rightarrow\tilde{M}$ be a  diffeomorphism.
Suppose $p\leq\tilde{p} < \infty$ and for any $\omega\in L^{\tilde{p}}(\tilde{M},\Lambda^{k})$ we
have 
\[
\left\Vert  f^{*}\omega\right\Vert _{L^{p}(M,\Lambda^{k})}\leq\left(K_{\tilde{p},t,k}( f)\right)^{1/\tilde{p}}\left\Vert \omega\right\Vert _{L^{\tilde{p}}(\tilde{M},\Lambda^{k})}
\]
where $t = \frac{p}{\tilde{p}-p}$. 
In particular if $f \in \BD^k_{(\tilde{p},t)}(M,\tilde{M})$, then the operator
\[
 f^{*}:L^{\tilde{p}}(\tilde{M},\Lambda^{k})\rightarrow L^{p}(M,\Lambda^{k})
\]
is bounded.
\end{proposition}

\begin{proof}
Without loss of generality we can suppose that $J_f(x)>0$.
Using the fact that $\left| (f^{*}\omega)_{x}\right|
\leq   \sigma_{k}(f,x) \cdot  \left| \omega_{f(x)}\right|$, we have 
\begin{align*} 
\left\Vert  f^{*}\omega\right\Vert _{L^{p}(M,\Lambda^{k})}^{p} &=
\int_{M}\left| (f^{*}\omega)_{x}\right|^{p}dx
\leq \int_{M}\left(\sigma_{k}( f,x)\right)^{p}\left|\omega_{f(x)}\right|^{p}dx
   \\     &  \leq
\int_{M}\left\{\left(\sigma_{k}( f,x)\, J_{ f}^{-1/\tilde{p}}(x)\right)^{p} \cdot \left(\left|\omega_{f(x)}\right|J_{ f}^{1/\tilde{p}}(x)\right)^{p}\right\}dx.
\end{align*}
Using H\"older's inequality for $s=\frac{\tilde{p}}{\tilde{p}-p}$ and
$s'=\frac{\tilde{p}}{p}$ (so that $\frac{1}{s}+\frac{1}{s'}=1$), and the change of variable formula, we obtain
\begin{align*}
\left\Vert  f^{*}\omega\right\Vert _{L^{p}(M,\Lambda^{k})}^{p}
 & \leq
\left(\int_{M}\left(\sigma_{k}^{\tilde{p}}( f,x)J_{ f}^{-1}(x)\right)^{\frac{p}{\tilde{p}-p}}dx\right)^{\frac{\tilde{p}-p}{\tilde{p}}} \cdot
\left(\int_{M}\left(\left|\omega_{f(x)}\right|^{\tilde{p}} J_{f}(x)\right)dx\right)^{\frac{p}{\tilde{p}}}
     \\  &  \leq
\left(K_{\tilde{p},t,k}( f)\right)^{\frac{p}{\tilde{p}}}\left(\int_{\widetilde{M}}\left|\omega_{y} \right|^{\tilde{p}}dy\right)^{\frac{p}{\tilde{p}}},
\end{align*}
that is 
\[
\left\Vert  f^{*}\omega\right\Vert _{L^{p}(M,\Lambda^{k})}  
\leq
\left(K_{\tilde{p},t,k}( f)\right)^{1/\tilde{p}}\left\Vert \omega\right\Vert _{L^{\tilde{p}}(\tilde{M},\Lambda^{k})}.
\]
\end{proof}

\bigskip

\textbf{Remark} Every diffeomorphism belongs to the class $BD_{1,\infty}^n$, i.e. $BD_{1,\infty}^n(M,\tilde{M})=  \text{Diff}(M,\tilde{M})$.The previous proposition states in particular the well known fact that the condition for an $n$-form to be integrable is invariant under diffeomorphism and therefore independant of the choice of a Riemannian metric.

\bigskip 

The next proposition describes the inverse of diffeomorphisms in $BD_{s,t}^k$.

\begin{proposition}\label{prop.BDMinv}
Let   $f:M\rightarrow\tilde{M}$ be a diffeomorphism,
$0\leq m \leq n$. Let  $1 \leq  \alpha  <\infty$ and $0 <\beta  \leq \infty$ with  $ \beta (\alpha-1) > 1$. 
Then the equivalence 
$$
  f^{-1} \in \BD_{(\alpha, \beta)}^m(\tilde{M},M) 
  \quad  \Leftrightarrow \quad
    f \in \BD_{(s,t)}^{n-m}(M, \tilde{M}) 
$$ 
holds if and only if   
\begin{equation}\label{eq.BDMinv}
 s= \frac{\alpha}{\alpha-1-\frac{1}{\beta}}  \qquad   \text{and}  \qquad
  t = \beta (\alpha-1) - 1.
\end{equation}
\end{proposition}

\medskip

\begin{proof} Without loss of generality we can suppose that $J(f,x)>0$.

 Assume first  that $\beta < \infty$, then the condition  $f^{-1} \in \BD_{(\alpha, \beta)}^m(\tilde{M},M)$  means that 
\[
\int_{\widetilde{M}}\left\{\sigma_{m}^{\alpha}( f^{-1},y)\, 
J_{ f^{-1}}^{-1}(y)\right\}^{\beta}dy<\infty.
\]
By the lemma \ref{lem:lin1},  we have 
\begin{equation}\label{eq:d2a}
 \sigma_{m}( f^{-1}, f(x))=\frac{\sigma_{n-m}( f,x)}{J_{ f}(x)}
\end{equation}
at $y= f(x)$ and for any $0\leq m\leq n$. 
Using  the relations  (\ref{eq.BDMinv}), which can be rewritten as 
\begin{equation*}\label{}
 \alpha \beta = s t = t+\beta +1,
\end{equation*}
together with  the change of variable formula with the 
standard relations $dy=J_{ f}(x)dx$, $J_{ f^{-1}}( f(x))=J_{ f}^{-1}(x)$, we can rewrite the latter integral as 
\begin{align*}
\int_{M}\left\{\left(\frac{\sigma_{n-m}( f,x)}{J_{ f}(x)}\right)^{\alpha}\, J_{ f}(x)\right\}^{\beta}J_{ f}(x)\, dx 
& =
\int_{M} \left(\sigma_{n-m}(f,x)\right)^{\alpha\beta}\,
\left(J_{f}(x)\right)^{1+\beta-\alpha\beta}  \, dx.
   \\
    &=   \int_{M}\left\{ \left(\sigma_{n-m}(f,x)\right)^{s}\,
\left(J_{f}(x)\right)^{-1} \right\}^{t} \, dx.
\end{align*}
This integral is finite if and only if $f \in \BD_{(s,t)}^{n-m}(M, \tilde{M})$.

\medskip

Assume now that $\beta = \infty$, then we also have $t= \infty$.
The condition  $f^{-1} \in \BD_{(\alpha,  \infty)}^m(\tilde{M},M)$  means in that case that 
\begin{equation}\label{condalbr}
 \sigma_{m}^{\alpha}( f^{-1})\, J_{ f^{-1}}^{-1}\quad \text{is uniformly bounded.}
\end{equation}
Using the relation  $s = \frac{\alpha}{\alpha-1}$, the equation (\ref{eq:d2a})  and  $J_{f^{-1}} = J_{ f}^{-1}$, we have 
\begin{align*}
 \left({\sigma_{m}(f^{-1})}\right)^{\alpha} J_{f^{-1}}^{-1}
 & = \left({\sigma_{m}(f^{-1})}\right)^{\alpha} J_f 
=
\left({\sigma_{m}(f^{-1})}{J_{ f}} \right)^{\alpha} J_f^{1-\alpha}
=
 \left\{ \sigma_{n-m}^{s}(f)\, J_{f}^{-1}   \right\}^{\alpha-1}
\end{align*}
Thus (\ref{condalbr}) holds if and only if  $\sigma_{n-m}^{s}(f)\, J_{f}^{-1}$ is bounded, i.e. 
$f \in \BD_{(s,t)}^{n-m}(M, \tilde{M})$.

\end{proof}

\bigskip

\begin{corollary}\label{cor:CrBMD2} 
If   $\tilde{q}\leq q$ and the  diffeomorphism $f$ belongs to  $\BD_{(s,t)}^{n-m}(M, \tilde{M})$
with
\begin{equation*}\label{}
 s = \frac{\tilde{q}}{\tilde{q}-1}, \qquad
 t =  \frac{q(\tilde{q}-1)}{q-\tilde{q}},
\end{equation*}
 then the operator 
\[
 f_{*}:L^{q}(M,\Lambda^{m})\rightarrow L^{\tilde{q}}(\tilde{M},\Lambda^{m})
\]
is bounded.
\end{corollary}

\begin{proof}
This follows immediately from  Proposition \ref{pro:prBMD1} and the previous proposition 
with $\alpha = q$ and $\beta =  \frac{\tilde{q}}{q-\tilde{q}}$.

\end{proof}

\bigskip

\begin{corollary}\label{cor:CrBMD3} 
If the diffeomorphism $f:M \to \tilde {M}$ satisfies $f\in  BD_{(q,\infty)}^{k}(M, \tilde{M}) \cap \BD_{(q',\infty)}^{n-k}(M, \tilde{M})$
with $q'=\frac{q}{q-1}$
then $f^{*}:L^{q}(\tilde{M},\Lambda^{k})\rightarrow L^{q}(M,\Lambda^{k})$ is an isomorphism 
\end{corollary}

\textbf{Proof}  It follows at once from the  Propositions \ref{prop.BDMinv} and \ref{pro:prBMD1}.

\qed

\section{Relation with quasiconformal and bilipschitz diffeomorphisms.}

Recall that an orientation preserving diffeomorphism\footnote{It is usual, and important, to consider not only diffeomorphisms, but more generally homeomorphisms in $W^{1,n}_{loc}$ when defining quasiconformal maps. In our present context, diffeomorphisms are sufficient, see however the discussion in section 8} $f : (M,g) \to (\tilde{M},\tilde{g})$,  between two oriented $n$-dimensional Riemannian manifolds is said to be \emph{quasiconformal} if
$$
 \frac{|df|^n}{J_f} \in L^{\infty}(M).
$$

\bigskip

\begin{lemma}\label{lem.qc}
 For the diffeomorphism $f : (M,g) \to (\tilde{M},\tilde{g})$, the following properties are equivalent
 \begin{enumerate}[(i.)]
  \item $f$ is quasiconformal;
  \item $f^{-1}$ is quasi-conformal;
  \item If  \  $\lambda_1(x),\lambda_2(x), \cdots , \lambda_n(x)$ are  the principal dilation coefficients
of $df_x$, then 
 $$
   \sup_{x\in M} \frac{\max \{\lambda_1(x),\lambda_2(x), \cdots , \lambda_n(x)\} }
   {\min \{\lambda_1(x),\lambda_2(x), \cdots , \lambda_n(x)\} }  <  \infty.
  $$
\end{enumerate}
\end{lemma}

The proof of this lemma is standard and easy.

\bigskip                                                                                      

Let us denote by $\QC(M,\tilde{M})$ the class of all quasiconformal diffeomorphisms, it is clear that $\QC(M,\tilde{M}) = \BD^1_{n,\infty}(M,\tilde{M})$, but, more generally:

\begin{proposition} We have
$$\QC(M,\tilde{M}) = \BD^k_{\frac{n}{k},\infty}(M,\tilde{M})$$
for any $1\leq k \leq n-1$.
\end{proposition}
 
\medskip
 
\textbf{Proof}  Suppose that  $f : (M,g) \to (\tilde{M},\tilde{g})$ is quasiconformal. Let us assume that  
$\lambda_1 \leq \lambda_2 \leq \cdots  \leq \lambda_n$, then 
by condition (iii) of the previous lemma,
there exits a constant $C$ such that
$$\sigma_k(f,x) \leq C \cdot (\lambda_1(x))^k.$$
Since $J_f = \lambda_1 \cdot \lambda_2  \cdot \cdots  \cdot \lambda_n$, we have 
$$
  \frac{(\sigma_k(f))^{n/k}}{J_{f}}  \leq  C \cdot  \frac{(\lambda_1^k)^{n/k}}{J_{f}}
  \leq  C \cdot  \frac{(\lambda_1)^{n}}{( \lambda_1 \cdot \lambda_2  \cdot \cdots  \cdot \lambda_n)}
  \leq  C,
$$
i.e. $f \in \BD^k_{\frac{n}{k},\infty}(M,\tilde{M})$. We have thus shown that $\QC(M,\tilde{M}) \subset \BD^k_{\frac{n}{k},\infty}(M,\tilde{M})$.

\medskip

To prove the converse inclusion, we distinguish three cases : $k =  \frac{n}{2}$, \ $1\leq k < \frac{n}{2}$ and 
$\frac{n}{2} < k < n$.

Let us first assume that $k =  \frac{n}{2}$, then we have 
$$
 \frac{\lambda_{n }}{\lambda_{1}} \leq   \frac{ (\lambda_{n-k+1}  \cdot \cdots  \cdot \lambda_{n})}{ (\lambda_{1}  \cdot \cdots  \cdot \lambda_{k}) } \leq 
  \frac{ (\lambda_{n-k+1}  \cdot \cdots  \cdot \lambda_{n})^2}{ (\lambda_{1}  \cdot \cdots  \cdot \lambda_{k})(\lambda_{n-k+1}  \cdot \cdots  \cdot \lambda_{n}) } \leq 
 \frac{(\sigma_k(f))^{2}}{J_{f}}, 
$$
 which implies  that 
$\BD^{n/2}_{2,\infty}(M,\tilde{M}) \subset \QC(M,\tilde{M}).$

\medskip

Assume now that   $1\leq k < \frac{n}{2}$,  i.e.  $k+1 \leq n-k$. Observe that
$$
 (\lambda_{k+1}  \cdot \cdots  \cdot \lambda_{n-k})  \leq (\lambda_{n-k})^{n-2k} \leq 
  (\lambda_{n-k+1}  \cdot \cdots  \cdot \lambda_{n})^{(n-2k)/k},
$$
therefore
\begin{align*}
 J_f  & = (\lambda_1 \cdot \lambda_2  \cdot \cdots  \cdot \lambda_n)
    \\ &   =  (\lambda_{1}  \cdot \cdots  \cdot \lambda_{k})  (\lambda_{k+1}  \cdot \cdots  \cdot \lambda_{n-k})  (\lambda_{n-k+1}  \cdot \cdots  \cdot \lambda_{n})     
   \\ &  \leq  (\lambda_{1}  \cdot \cdots  \cdot \lambda_{k})  (\lambda_{n-k+1}  \cdot \cdots  \cdot \lambda_{n})^{\frac{n-2k}{k}+1}
    \\ & = (\lambda_{1}  \cdot \cdots  \cdot \lambda_{k})  (\lambda_{n-k+1}  \cdot \cdots  \cdot \lambda_{n})^{\frac{n}{k}-1}.
\end{align*}
Because $ \sigma_k \geq    \lambda_{n-k+1}  \cdot \cdots  \cdot \lambda_n$, we  have from the previous inequality
\begin{align*}
 \frac{(\sigma_k(f))^{n/k}}{J_{f}}  & \geq \frac{ (\lambda_{n-k+1}  \cdot \cdots  \cdot \lambda_{n})^{\frac{n}{k}}}{J_{f}}
 \geq  \frac{ (\lambda_{n-k+1}  \cdot \cdots  \cdot \lambda_{n})}{ (\lambda_{1}  \cdot \cdots  \cdot \lambda_{k}) }.
\end{align*}
Since 
$$
 \frac{\lambda_{n-k+1}}{\lambda_{k}} , \ \frac{\lambda_{n-k+2}}{\lambda_{k-1}} , \dots , \frac{\lambda_{n-1}}{\lambda_{2}} \,  \geq 1,
$$
we finally have 
$$
 \frac{\lambda_{n }}{\lambda_{1}} \leq   \frac{ (\lambda_{n}  \cdot \cdots  \cdot \lambda_{n-k+1})}{ (\lambda_{1}  \cdot \cdots  \cdot \lambda_{k}) } \leq \frac{(\sigma_k(f))^{n/k}}{J_f}, 
$$
from which follows that 
$\BD^k_{\frac{n}{k},\infty}(M,\tilde{M}) \subset \QC(M,\tilde{M}).$

\medskip

If $k > \frac{n}{2}$, then $n-k < \frac{n}{2}$ and we have from the previous argument and Proposition \ref{prop.BDMinv}
$$
 f^{-1} \in \BD^{n-k}_{\frac{n}{n-k},\infty}(\tilde{M},M)  \subset \QC(\tilde{M},M),
$$
and we deduce from  lemma \ref{lem.qc}  that $f\in \QC(M,\tilde{M})$.

\qed

\bigskip

The next result relates our class of maps to bilipschitz ones.

\begin{proposition}
If   $f\in  BD_{(q,\infty)}^{k}(M, \tilde{M}) \cap \BD_{(q',\infty)}^{n-k}(M, \tilde{M})$ with $q' = \frac{q}{q-1}$, then 
$f$ is quasiconformal. Furthermore if $q \neq \frac{n}{k}$, then $f$ is bilipschitz.
\end{proposition}

\bigskip

\textbf{Proof}
Using the same notations and convention as in the previous proof, we have  \begin{align*}
  \frac{\lambda_{n }}{\lambda_{1}}   
  & \leq   \frac{ (\lambda_{n}  \cdot \cdots  \cdot \lambda_{n-k+1})}{ (\lambda_{1}  \cdot \cdots  \cdot \lambda_{k}) } 
= \ 
 \frac{ (\lambda_{n-k+1}  \cdot \cdots  \cdot \lambda_{n}) (\lambda_{k+1}  \cdot \cdots  \cdot \lambda_{n})}{ (\lambda_{1}  \cdot \cdots  \cdot \lambda_{k}) (\lambda_{k+1}  \cdot \cdots  \cdot \lambda_{n})}
 \\ & \leq  \frac{\sigma_k(f)\cdot \sigma_{n-k}(f)}{J_f}
 = \left(\frac{\left(\sigma_k(f)\right)^q}{J_f}\right)^{\frac{1}{q}}  \left(\frac{\left(\sigma_{n-k}(f)\right)^{q'}}{J_f}\right)^{\frac{1}{q'}},
\end{align*}
because $\frac{1}{q}+\frac{1}{q'} = 1$. It follows from this computation that any map $f$ in 
$BD_{(q,\infty)}^{k}(M, \tilde{M}) \cap \BD_{(q',\infty)}^{n-k}(M, \tilde{M})$  is quasiconformal. 

\medskip

We now prove that $f$ is bilipschitz if $q \neq  \frac{n}{k}$:
Because $f$ s quasiconformal, there exists a constant $c$ such that $\lambda_n \leq c \cdot \lambda_1$.
Since $\lambda_1^k \leq \sigma_k(f)$ and $J_f \leq \lambda_n^k$, we have 
$$
 |df|^{qk-n} =  \lambda_n^{qk-n} \leq   \frac{(c\lambda_1)^{kq}}{\lambda_n^n} \leq c^{kq}\cdot \frac{\left(\sigma_k(f)\right)^q}{J_f},
$$
this implies that any quasiconformal map in $BD_{(q,\infty)}^{k}(M, \tilde{M})$ is lipschitz if $qk>n$.
If $qk<n$, then $q'(n-k)<n$ and the same argument shows that 
any quasiconformal map in $BD_{(q',\infty)}^{n-k}(M, \tilde{M})$ is lipschitz.
Thus any $f\in  BD_{(q,\infty)}^{k}(M, \tilde{M}) \cap \BD_{(q',\infty)}^{n-k}(M, \tilde{M})$ with  $q \neq \frac{n}{k}$  is lipschitz.
But Proposition\ref{prop.BDMinv} implies that $f^{-1} \in  BD_{(q,\infty)}^{k}(\tilde{M},M) \cap \BD_{(q',\infty)}^{n-k}(\tilde{M},M)$,
hence $f^{-1}$ is also a lipshitz map if $q \neq \frac{n}{k}$.

\qed

\bigskip

\textbf{An open question.}  The previous result  and the Corollary \ref{cor:CrBMD3} suggest the following question:
\emph{Suppose a diffeomorphism $f:M \to \tilde {M}$ induces an isomorphism $f^{*}:L^{q}(\tilde{M},\Lambda^{k})\rightarrow L^{q}(M,\Lambda^{k})$. 
Can we conclude that $f$ is quasiconformal for $q=\frac{n}{k}$ and bilipshitz otherwise?}

\medskip

If $k=1$, the answer to the above question is positive, see  \cite{GGR,GRo,VG76,VU}.

\medskip

For a more complete discussion of quasiconformal maps in the context of differential forms, we refer to \cite{GT2007} .

\section{$L_{q,p}$-cohomology and $\BD$-diffeomorphisms.}

Combining the results of the two previous sections,  we obtain the following theorem.

\begin{theorem}\label{thm:CA}
Suppose $p\leq\tilde{p}<\infty$, and let
$ f:M\rightarrow\tilde{M}$ be a diffeomorphism of the class $\BD^k_{(\tilde{p},t)}(M,\tilde{M})$
where $t = \frac{p}{\tilde{p}-p}$. 
Then the  following holds:
\begin{enumerate}[A.)]
  \item $f^{*}:L^{\tilde{p}}(\tilde{M},\Lambda^{k})\rightarrow L^{p}(M,\Lambda^{k})$ is a bounded operator
  and $f^*(Z_{\tilde{p}}^{k}(\tilde{M})) \subset Z_{{p}}^{k}({M})$.
  \\
 \item If  $q\geq \tilde{q}>1$ and 
  $f\in \BD_{(\tilde{q}',r)}^{n-k+1}(M, \tilde{M}) \cap \BD^k_{(\tilde{p},t)}(M,\tilde{M})$
  with $\tilde{q}' = \frac{\tilde{q}}{\tilde{q}-1}, r =  \frac{q(\tilde{q}-1)}{q-\tilde{q}}$, 
 then 
    $\left[ f^{*}\omega\right]=0$ in $H_{q,p}^{k}(M)$ implies $\left[\omega\right]=0$ in $H_{\tilde{q},\tilde{p}}^{k}(\tilde{M})$
    (thus $H_{q,p}^{k}(M) = 0 \Rightarrow H_{\tilde{q},\tilde{p}}^{k}(\tilde{M})= 0$).
    \\
 \item If  $q\leq \tilde{q}$ and 
 $f\in \BD_{(\tilde{q},u)}^{k-1}(M, \tilde{M}) \cap \BD^k_{(\tilde{p},t)}(M,\tilde{M})$
    where $u = \frac{q}{\tilde{q}-q}$ and $t = \frac{p}{\tilde{p}-p}$, then 
   \vspace{.2cm}
  \begin{enumerate}[\quad a.)]
  \item $f^{*}:\Omega_{\tilde{q},\tilde{p}}^{k-1}(\tilde{M})\rightarrow\Omega_{q,p}^{k-1}(M)$
is a bounded operator, 
  \item $ f^{*}:H_{\tilde{q},\tilde{p}}^{k}(\tilde{M})\rightarrow H_{q,p}^{k}(M)$
is a well defined linear map,
  \item $ f^{*}:\overline{H}_{\tilde{q},\tilde{p}}^{k}(\tilde{M})\rightarrow \overline{H}_{q,p}^{k}(M)$
is a bounded operator. \\
\end{enumerate}
\end{enumerate}
\end{theorem}

\medskip

\begin{proof}
The statement  (A) follows immediately from Proposition \ref{pro:prBMD1} and the fact that $df^*\omega = f^*d\omega$,
whereas the assertion (B) follows from Proposition \ref{pro:prBMD1}, Proposition \ref{prop.BDMinv} and Theorem \ref{thm:thA}. \  Finally, the  property (C) follows from  Proposition \ref{pro:prBMD1} and Theorem \ref{thm:thB}.

\end{proof}

\bigskip

 Part (C) of the Theorem gives us  sufficient conditions on a map $f$ to have a functorial behavior in 
 $L_{q,p}$-cohomology.

\section{Some examples}

In this section, we  show how Theorem \ref{thm:CA} can be used to produce vanishing and non vanishing results for the $L_{q,p}$-cohomology of some specific manifolds. The calculations can be quite delicate, even for familiar Riemannian manifolds, and here we only give two simple examples, without trying
to obtain optimal results. 

\subsection{A manifold with a cusp}
  
\medskip

Let us consider the Riemannian manifold $(\tilde{M},g)$ such that $M$ is diffeomorphic to $\r^n$ and $\tilde{g}$
is a Riemannian metric such that in polar coordinates, we have
$$
 \tilde{g} = dr^2 +  e^{-2r} \cdot h
$$
for large enough $r$, where $h$ denotes the standard metric on the sphere  $\mathbb{S}^{n-1}$.
Let us also consider the identity map $f : \r^n \to \tilde{M}$, where $\r^n$ is given its standard euclidean metric, which writes in polar coordinates as
$$
 ds^2 = dr^2 +  r^2 \cdot h.
$$

\begin{proposition}\label{Pcusp1}
If $s>\frac{n-1}{m-1}$, then the above map $f : \r^n \to \tilde{M}$
belongs to the class $\BD^m_{s,t}(\mathbb{R}^n, \tilde{M})$ for any $0<t \leq \infty$.
\end{proposition}

\medskip

\textbf{Proof}  For $r$ large enough, we have the following principal dilatation coefficients for $f$:
$$
 \lambda_1 = 1, \qquad \lambda_2 = \lambda_3 = \cdots = \lambda_n = \frac{e^{-r}}{r}.
$$
In particular  $J_f = \left( \frac{e^{-r}}{r}\right)^{n-1}$ and 
$$
 \sigma_m(f) = \left( \frac{e^{-r}}{r}\right)^{m} + \binom{n-1}{m-1} \left( \frac{e^{-r}}{r}\right)^{m-1}
  \leq  C_1 \left( \frac{e^{-r}}{r}\right)^{m-1}.
$$
and thus
$$
 \frac{ \left(\sigma_m(f) \right)^{s}}{J_f}  \leq  C_2 \,   \left( \frac{e^{-r}}{r}\right)^{s(m-1) - (n-1)}
$$
outside a compact set in $\r^n$.
Therefore $ \int_{\mathbb{R}^n}   \left( \frac{ \left(\sigma_m(f) \right)^{s}}{J_f} \right)^{t} dx < \infty$
if and only if 
$$
\int_{1}^{\infty}   \left( \frac{e^{-r}}{r}\right)^{t(s(m-1) - (n-1))} \cdot r^{n-1} dr  < \infty
 $$
which is the case when  $s\geq \frac{n-1}{m-1}$. This implies that
$f \in \BD^m_{s,t}(\mathbb{R}^n, \tilde{M})$ for any  $0 < t < \infty$.

It is also clear that $f \in \BD^m_{s,\infty}(\mathbb{R}^n, \tilde{M})$, since $ \frac{ \left(\sigma_m(f) \right)^{s}}{J_f}$ is bounded when  $s\geq \frac{n-1}{m-1}$.

\qed

\bigskip

\begin{corollary}
 \ If $\ds \tilde{q} <\frac{n-1}{k-1} < \tilde{p}$, then $H^k_{\tilde{q},\tilde{p}}(\tilde{M})  = 0$.
\end{corollary}

\textbf{Proof} \\ We will use Theorem \ref{thm:CA}(B) with the previous Proposition.
We have  $f\in  \BD^k_{(\tilde{p},t)}(\r^n,\tilde{M})$  for any $t>0$, since we have 
$\tilde{p} > \frac{n-1}{k-1}$ by hypothesis.
We also have  $f\in \BD_{(\tilde{q}',r)}^{n-k+1}(\r^n, \tilde{M})$ 
if $\tilde{q}'  >\frac{n-1}{n-k}$. But this inequality is equivalent to
$$
  \tilde{q} = \frac{\tilde{q}'}{\tilde{q}'-1} <\frac{n-1}{k-1},
$$
and this also holds by hypothesis.
We thus have  $f\in\BD_{(\tilde{q}',r)}^{n-k+1}(\r^n, \tilde{M}) \cap \BD^k_{(\tilde{p},t)}(\r^n, \tilde{M})$
for any $\tilde{q} <\frac{n-1}{k-1} < \tilde{p}$.

Let us now set $p=\frac{n}{k}$ and $q=\frac{n}{k-1}$, and observe that $p \leq \frac{n-1}{k-1}$,
hence $p \leq \tilde{p}$ and $q \geq \frac{n-1}{k-1}$, hence $q \geq \tilde{q}$.

In \cite{T2008a}, it is  proved that $H^k_{q,p}(\mathbb{R}^n)\neq 0$ if $p=\frac{n}{k}$ and $q=\frac{n}{k-1}$.
Therefore by Theorem \ref{thm:CA} we have  $H^k_{\tilde{q},\tilde{p}}(\tilde{M})  = 0$
 for any $\ds \tilde{q} <\frac{n-1}{k-1} < \tilde{p}$.

\qed

\subsection{The hyperbolic space}
  
\medskip

Let us denote by  $\mathbb{H}^n$ the hyperbolic space of dimension $n$.
Recall that $\mathbb{H}^n$ can be described in polar coordinate
as follow:
$$\mathbb{H}^n = [0, \infty) \times \mathbb{S}^{n-1} / (\{0\} \times \mathbb{S}^{n-1}),$$
with the Riemannian metric 
$$
 g = dr^2 + \sinh(r)^2 h,
$$
where $h$ is the standard metric on the sphere $ \mathbb{S}^{n-1}$. Likewise, the euclidean space $\r^n$ 
is given by $\mathbb{R}^n = [0, \infty) \times \mathbb{S}^{n-1} / (\{0\} \times \mathbb{S}^{n-1})$, 
with the Riemannian metric 
$
 ds^2 = dr^2 + r^2 h.
$

Let us consider the identity map $f : \mathbb{H}^n \to \mathbb{R}^n$ (which is, from an intrinsic viewpoint, the inverse of the
exponential map $\exp_x : T_x\mathbb{H}^n =\mathbb{R}^n\to \mathbb{H}^n$).

\begin{proposition}\label{PH}
 The above map $f : \mathbb{H}^n \to \r^n$ belongs to the class $\BD^m_{s,t}(\mathbb{H}^n, \r^n)$ for $1\leq s< \infty$, $0<t \leq \infty $
 if and only if 
\begin{equation}\label{eq.condHn}
  s > \frac{n-1}{m-1}\left(1 + \frac{1}{t} \right)
\end{equation}
 and belongs $\BD^m_{s,\infty}(\mathbb{H}^n, \r^n)$  if and only if 
 \begin{equation}\label{eq.condHn1}
  s \geq \frac{n-1}{m-1}.
\end{equation}
\end{proposition}

\medskip

\textbf{Proof} We clearly have the following principal dilatation coefficients for $f$:
$$
 \lambda_1 = 1, \qquad \lambda_2 = \lambda_2 = \cdots = \lambda_n = \frac{r}{\sinh(r)}.
$$

Therefore
$$
 \sigma_m(f) = \left( \frac{r}{\sinh(r)} \right)^{m} + \binom{n-1}{m-1} \left( \frac{r}{\sinh(r)} \right)^{m-1}
  \leq \text{const. }  \left( \frac{r}{\sinh(r)} \right)^{m-1},
$$
and thus
$$
 \frac{ \left(\sigma_m(f) \right)^{s}}{J_f}  \leq  C \,   \left( \frac{r}{\sinh(r)} \right)^{s(m-1) - (n-1)},
$$
It follows that $f \in \BD^m_{s,\infty}(\mathbb{H}^n, \r^n)$ if and only if $s\geq \frac{n-1}{m-1}$.
Likewise, $f \in \BD^m_{s,t}(\mathbb{H}^n, \r^n)$ for some $t < \infty$ when the integral
$$
 \int_{\mathbb{H}^n}   \left( \frac{ \left(\sigma_m(f) \right)^{s}}{J_f} \right)^{t}
 \leq \mathrm{const.}  \int_{0}^{\infty}  \left( \frac{r}{\sinh(r)} \right)^{t(s(m-1) - (n-1))}\cdot (\sinh(r))^{n-1} dr 
 $$
is finite. This is the case if and only if 
$$
 t(s(m-1) - (n-1)) > (n-1).
$$
And this inequality is equivalent to (\ref{eq.condHn}).

\qed

\begin{corollary}
 \ If $\ds q <\frac{n-1}{k-1} < p$, then $H^k_{q,p}(\mathbb{H}^n)  \neq 0$.
\end{corollary}

\textbf{Proof}  We will use Theorem \ref{thm:CA} with the previous Proposition.
We have  $f\in  \BD^k_{(\tilde{p},t)}(\mathbb{H}^n, \r^n)$ with $t = \frac{p}{\tilde{p}-p}$ if and only if 
$$
 \tilde{p} >\frac{n-1}{k-1}\left(1 + \frac{1}{t} \right) = 
 \frac{n-1}{k-1}\left(1 +  \frac{\tilde{p}-p}{p} \right) =
  \frac{n-1}{k-1}\cdot \frac{\tilde{p}}{p},
$$
i.e. 
$$p>\frac{n-1}{k-1}.$$

\medskip

Likewise,  $f\in \BD_{(\tilde{q}',r)}^{n-k+1}(\mathbb{H}^n, \r^n)$ with 
$\tilde{q}' = \frac{\tilde{q}}{\tilde{q}-1}, r =  \frac{q(\tilde{q}-1)}{q-\tilde{q}}$ if and only if
$$
 \tilde{q}' =   \frac{\tilde{q}}{\tilde{q}-1} >\frac{n-1}{n-k}\left(1 + \frac{1}{r} \right) = 
 \frac{n-1}{n-k}\left(1 +  \frac{q-\tilde{q}}{q(\tilde{q}-1)} \right).
$$
This inequality is equivalent to 
$$
 \tilde{q} > \frac{n-1}{n-k}\left((\tilde{q}-1) + \frac{q-\tilde{q}}{q}\right) 
 = \frac{n-1}{n-k} \left(1 - \frac{1}{q}\right) \tilde{q},
$$
or, finally
$$
 q < \frac{n-1}{k-1}.
$$

We proved that $f\in\BD_{(\tilde{q}',r)}^{n-k+1}(\mathbb{H}^n,\mathbb{R}^n) \cap \BD^k_{(\tilde{p},t)}(\mathbb{H}^n,\mathbb{R}^n)$ whenever
\begin{equation} \label{ccond}
  q < \frac{n-1}{k-1}<p, \quad  p \leq \tilde{p}, \ \quad   q \geq \tilde{q},  \ \quad 
   t = \frac{p}{\tilde{p}-p}, \ \quad  r =  \frac{q(\tilde{q}-1)}{q-\tilde{q}}.
\end{equation}
In \cite{T2008a}, it is  proved that $H^k_{\tilde{q},\tilde{p}}(\mathbb{R}^n)\neq 0$ unless $\frac{1}{\tilde{p}} - \frac{1}{\tilde{q}} = \frac{1}{n}$. 
Therefore one can choose some values of $\tilde{p},\tilde{q}$ compatible with the conditions (\ref{ccond}) and use Theorem \ref{thm:CA}  to conclude that $H^k_{q,p}(\mathbb{H}^n)\neq 0$ for any 
$q < \frac{n-1}{k-1}<p.$

\qed

\bigskip

The result given in the previous Theorem is not optimal and we shall discuss the $L_{q,p}$-cohomology
of the hyperbolic space and other manifolds with negative curvature in a another paper.

\section{Non-smooth mappings}

We have formulated our results for diffeomorphisms, but it is clear that  the Definition \ref{def.BMD} makes sense
for wider classes of maps such as Sobolev maps in $W^{1,1}_{loc}$ or maps which are approximately differentiable 
almost everywhere, we can thus consider the class of $W^{1,1}_{loc}$ homeomorphisms with bounded mean distortion.
It is then natural and important to wonder whether our results still hold in this wider context.

\medskip

Unfortunately, there is no elementary  answer to this question. A careful look at our arguments show that we have used
the following properties of diffeomorphisms:
\begin{enumerate}[i.)]
  \item The \emph{change of variables formula in integrals} : $\int_M u(f(x)) \, J_f(x) dx = \int_{\tilde{M}} u(y) \, dy$ \
 in Proposition \ref{pro:prBMD1}).
\item The \emph{change of variables formula for the inverse map} : $\int_M u(f(x)) \, J_f(x) dx = \int_{\tilde{M}} u(y) \, dy$, \
 this is implicitly used in  Corollary \ref{cor:CrBMD2}.
  \item The naturality of the exterior differential $df^*\omega = f^*d\omega$  is used everywhere.
\end{enumerate}

\medskip

The change  of variables formula in integrals holds for a homeomorphism $f$ in $W^{1,1}_{loc}$ provided we assume the 
\emph{Luzin $(N)$ condition} to hold. This condition states that a subset of zero measure in $M$ is mapped by $f$ onto a set of
zero measure in $\tilde{M}$. The map change of variables formula for the inverse map $f^{-1}$ holds if the Luzin $(N^{-1})$ condition holds,
that is the inverse image of  subset of zero measure also has  zero measure. The Luzin condition is widely studied in the literature (see, for example, \cite{VG76,HK1993,KKM2001,HK2005}).
Concerning the naturality of the exterior differential, we refer to \cite{GT2008}.

\medskip

Let finally mention that for the special case of quasiconformal mapping, all these properties hold. The relation between the
theory of quasiconformal mappings and $L_{qp}$-cohomology is studied in   \cite{GT2007}.



\begin{thebibliography}{40}


\bibitem{ferrand}
Lelong-Ferrand L. \emph{Etude d'une classe d'applications
li\'{e}es \`{a} des homomorphismes d'alg\`{e}bres de fonctions et
g\'{e}n\'{e}ralisant les quasi-conformes.} Duke Math.
J.\textbf{40}, (1973) 163--186.

\bibitem{Gaf}     Gafa\"{\i}ti  K.\
\emph{Alg\`ebre de Royden et Hom\'eomorphismes \`{a} $p$-dilatation
born\'ee entre espaces m\'etriques mesur\'es.} Th\`{e}se, EPFL
Lausanne (2001).
\bibitem{GR} Gol'dshtein V.M., Reshetnyak Yu.G., \emph{ Quasiconformal Mappings and Sobolev
Spaces,} Kluwer Academic Publishers, Dordrecht/Boston/London, 1990. 
\bibitem{GK1} Gol'dshtein V.M., Kuz'minov V.I., Shvedov  I.A., \emph{Differential forms on Lipschitz Manifolds,}  Siberian Math. Journal, {\textbf 23}, No 2
(1982), 16-30. English translation in: Siberian Math. J. {\textbf 23}, No 2
(1982), 151-161.

\bibitem{GK2} Gol'dshtein V.M., Kuz'minov V.I., Shvedov  I.A., \emph {$L_p$-cohomology of warped cylinder,}  Siberian Math. Journal, {\textbf 31}, No 6
(1990), 55-63. English translation in: Siberian Math. J. {\textbf 31}, No 6
(1990), 716-727.
\bibitem{GK3} V. M. Gol'dshtein, V.I. Kuz'minov, I.A.Shvedov \emph{Dual spaces of
Spaces of Differential Forms} Siberian Math. Journal, (1986), \textbf{54}, No 1, 35-43.

\bibitem{GGR} Gol'dshtein V., Gurov L., Romanov A., \emph{ Homeomorphisms that induce Monomorphisms of Sobolev Spaces,} Israel Journal of Math.(1995), {\textbf 91,} No 1, 31--60. 
\bibitem{GG} Gol'dshtein V., Gurov L., \emph{Applications of change of variable operators for exact embedding theorems,} Integral equations and operator theory. (1994), {\textbf 19,} No 1, 1--24.

\bibitem{GRo} Gol'dshtein V.M., Romanov A.S. ,
\emph{Transformations that preserve Sobolev spaces}, (1984), {\textbf 25}, No 3, 382-388. 

\bibitem{SOL}  Gol'dshtein V.  and  Troyanov M. \emph{The $L_{pq}$-cohomology
of $SOL$.} Annales de la Facult\'e des Sciences de Toulouse. Vol. Vii,
No 4, 1998.
\bibitem{GT2006} Gol'dshtein V. and Troyanov M., \emph{Sobolev Inequality for
Differential forms and $L_{q,p}$-cohomology,} Journal of Geom. Anal.(2006),
\textbf{16,} No 4,  597-631. 
\bibitem{GT2007} Gol'dshtein V. and Troyanov M., \emph{A conformal de Rham complex,} arXiv:0711.1286. 
\bibitem{GT2008} Gol'dshtein V. and Troyanov M., \emph{On the naturality of exterior  differential} arXiv:0801.4295. 
To appear in Math. Reports of the Canadian Academy of Science.
\bibitem{GU} Gol'dshtein V., Ukhlov A., \emph{Weighted Sobolev spaces and embedding theorems,} Transactions of Amer. Math. Soc. (to appear).
\bibitem{Grom}  Gromov M. \emph{Asymptotic invariants of infinite groups} in``Geometric
group theory, volume 2'' London Math. Soc. Lecture Notes \textbf{182,}
Cambridge University Press (1992).
\bibitem{HK1993} Heinonen J. Koskela P., \emph{Sobolev mappings with integrable distortion,} Arch. Rat.Mech. Anal. (1993), \textbf{125},81-97.
\bibitem{IS1993} Iwaniec T., Sverak V., \emph{On mappings with integrable dilatation,} Proc. Amer. Math. Soc., (1993), \textbf{118}, 181-188.
\bibitem{KKM2001} Kauhanen J., Koskela P., Maly j., \emph{Mappings of finite distortion: Discretness and openness,} Arch. Rat. Mech. Anal., (2001), \textbf{160}, no 2, 135-151.
\bibitem{HK2005} Hencl S., Koskela P., \emph{Mapping of finite distortion: openess and discretness for quasilight mappings,} Ann. Inst. H. Poincar\'e, (2005), \textbf{22}, 331-342. 
\bibitem{kopylov2007} Kopylov Y.A.,   \emph{ $L_{q,p}$-cohomology and normal solvability}, Arch. Math. (2007), {\bf 89}, No 1, 87-96.
\bibitem{kopylov2008}  Kopylov Y.A.,   \emph{$L_{p,q}$-Cohomology of Warped Cylinders} 	arXiv:0803.3298v1
\bibitem{luck2002} L\"uck W.,  \emph{$L\sp 2$-invariants: theory and applications to geometry and $K$-theory,} Springer-Verlag, Berlin, 2002
\bibitem{MV1995} Manfredy J., Villamor E., \emph{Mappings with integral dilatation in higher dimension,} Bull. Amer. Math. Soc., (1995), \textbf{32}, no.2, 235-240.
\bibitem{MH} Marsden J. and Hughes T., 
\emph{Mathematical foundations of elasticity,}  
 Prentice-Hall (1983). 
\bibitem{Maz85}   Maz'ya V.G.  \emph{Sobolev Spaces.} Springer Verlag (1985).
\bibitem{Masha}
 Maz'ya V.G and Shaposhnikova T. \emph{Theory of Multipliers in
Spaces of Differentiable Functions.} Pitman (1985).
\bibitem{pansu97}
Pansu  P.\emph{Diff\'{e}omorphismes de }$p$\emph{-dilatation
born\'{e}es}. Ann. Acad. Sc. Fenn. \textbf{223} (1997) 475--506.
\bibitem{pansu99} Pansu P. \emph{Cohomologie $L^{p}$, espaces homog\`{e}nes et pincement.} Preprint, Orsay, 1999.
\bibitem{pansu2002} Pansu P. \emph{$L\sp p$-cohomology and pinching.} in  Rigidity in dynamics and geometry (Cambridge, 2000),  379--389, Springer, Berlin, 2002. 

\bibitem{pansu2007} Pansu P. 
\emph{Cohomologie $L\sp p$ en degr 1 des espaces homognes.}
J. Potential Anal. 27, 151-165 (2007).
\bibitem{pansu2008}  Pansu P. \emph{Cohomologie  $L\sp p$  et pincement} .     Comment. Math. Helvetici. (to appear)

\bibitem{Reiman}
Reiman M. 
 \emph{\"Uber harmonishe Kapazit\"at und quasikonforme Abbildungen
in Raum.} Comm. Math. Helv. \textbf{44} (1969)  284--307.
\bibitem{resh}
Reshetnyak Yu.G.    \emph{Space Mappings with Bounded distortion,}
Translations of Mathematical Monographs, (1985),  \textbf{73}, American Mathematical Society.
\bibitem{rick}  Rickman S.  \emph{Quasiregular mapping} Springer-Verlag, Berlin-Heidelberg-New York, 1993.
\bibitem{TV2002} Troyanov M.  and  Vodop'yanov S.K.  \emph{ Liouville type theorem for 
mappings with bounded co-distortion}. Annales de l'Institute Fourier. (2002), {\textbf 52}, No 6, 1754-1783.
\bibitem{T2008a}  Troyanov M.  . \emph{On the Hodge decomposition in $\r^n$}.  arXiv:0710.5414.
\bibitem{Vdp2000} Vodop'yanov   S.K.   \emph{Topological and geometrical properties of mappings with an integrable Jacobian in Sobolev classes,} Siberian math.J., (2000), \textbf{41},
no 4., 19-39.

\bibitem{VG76} Vodop'yanov S.K., Gol'dshtein V.M., \emph{Quasiconformal mappings and spaces of mfubctions with generalized first derivatives,}
Siberian math. J., (1976), \textbf{17}, no 3, 515-531.

\bibitem{VU} Vodop'yanov S.K., Ukhlov A.D. \emph{ Sobolev spaces and $(p,q)$-quasiconformal mappings of Carnot groups.} Siberian Math. J. (1998) {\textbf 39}, No 4, 776--795. 

\end{thebibliography}
\end{document}